\pdfoutput=1
\documentclass[11pt]{article}
\usepackage[T1]{fontenc}
\usepackage[margin=1in]{geometry}
\usepackage{amsmath,amssymb,amsthm,mathtools}
\usepackage{booktabs,graphicx,microtype,xcolor}
\usepackage{etoolbox,tikz}
\usepackage{hyperref}
\definecolor{navy}{HTML}{17324D}
\definecolor{teal}{HTML}{087F8C}
\hypersetup{
  colorlinks=true,linkcolor=navy,citecolor=navy,urlcolor=teal,
  pdftitle={Brownian covariances on an interval: a complete endpoint classification},
  pdfauthor={Lars Lau Raket}
}
\graphicspath{{figures/}}
\newtheorem{theorem}{Theorem}[section]
\newtheorem{proposition}[theorem]{Proposition}
\theoremstyle{definition}
\newtheorem{definition}[theorem]{Definition}
\newtheorem{example}[theorem]{Example}
\newcommand{\exampleendmark}{\,\tikz[baseline=-0.4ex]{\draw[line width=0.6pt] (0,0) circle (1.6pt);}}
\AtEndEnvironment{example}{\hfill\exampleendmark}
\DeclareMathOperator{\Var}{Var}
\DeclareMathOperator{\tr}{tr}
\DeclareMathOperator{\rank}{rank}
\DeclareMathOperator{\ran}{range}
\DeclareMathOperator{\dom}{dom}
\DeclareMathOperator{\diag}{diag}
\newcommand{\R}{\mathbb R}
\newcommand{\dd}{\,\mathrm d}
\newcommand{\T}{\mathsf T}
\newcommand{\bker}{k_{\mathrm D}}
\newcommand{\psd}{\mathbb S_+^2}
\newcommand{\pd}{\mathbb S_{++}^2}

\title{Brownian covariances on an interval:\\a complete endpoint classification}
\author{Lars Lau Raket}
\date{}

\begin{document}
\maketitle

\begin{abstract}
Brownian motion and the Brownian bridge have covariance kernels that are
Green's functions of the negative second derivative on the unit interval,
with mixed Dirichlet--Neumann and Dirichlet boundary conditions,
respectively. We characterize the Green's functions of all strictly positive
real self-adjoint realizations of this operator. The kernels are in
one-to-one correspondence with positive-semidefinite $2\times2$ endpoint
covariance matrices. We call the associated centered Gaussian processes
\emph{Brownian processes}. Each is a standard Brownian bridge plus the
affine interpolation of an independent Gaussian endpoint vector.
Singular endpoint covariance matrices give hard endpoint constraints, such
as the pinned start $X(0)=0$ of Brownian motion. This parametrization gives explicit operator domains and
energy spaces. On every mesh it gives an exact sparse precision matrix,
so the likelihood of $n$ observations can be evaluated in $O(n)$
operations. Nonnegative realizations with a null space are
treated separately through orthogonally normalized pseudoinverses, which
are proper Gaussian covariances.
\end{abstract}


\section{Introduction}\label{sec:introduction}

Brownian motion's covariance function is the Green's function of the negative
second derivative on $[0,1]$ with operator-domain conditions $f(0)=0$ and
$f'(1)=0$. The Brownian bridge's covariance corresponds
to the Green's function under $f(0)=f(1)=0$. This raises two questions: which other endpoint conditions give Green's functions that are covariance functions, and how are the resulting processes related to Brownian motion and the Brownian bridge?

We answer both questions through the covariance of the two endpoint
values, instead of listing boundary conditions one by one. All processes are
Gaussian and centered unless stated otherwise. It turns out that after subtracting the
affine interpolation of its endpoints, each process classified in Theorem~\ref{thm:classification}
is a standard Brownian bridge independent of those endpoints. Parametrizing
the endpoint covariance also includes hard endpoint constraints as finite
parameter values.

Self-adjoint boundary realizations are classical \cite{zettl2005}; their
connections with Gaussian processes and spline penalties are equally well
established \cite{kimeldorf1970,rasmussen2006,raket2021}. The companion
bestiary \cite{raketbestiary} develops examples and distinguishes changing
boundary realizations, recentering and conditioning. Markussen
\cite{markussen2013} used this correspondence for likelihood inference in
functional mixed-effects models, with Brownian motion and the Brownian
bridge among his examples. Markussen's construction uses prescribed boundary
conditions. Here every admissible choice
corresponds to an endpoint covariance matrix, so the boundary conditions
can be treated as unknown parameters. We classify the
full positive-semidefinite endpoint-covariance cone and treat noninvertible
realizations separately.
Section~\ref{sec:classification} proves the classification and describes
hard endpoint constraints. Section~\ref{sec:geometry} studies the geometry
of the parameter cone and identifies the Markov and stationary members.
Section~\ref{sec:mesh} shows that the sampled precision matrix is exact and
sparse on any mesh, and Section~\ref{sec:singular} treats nonnegative
realizations with a null space. Section~\ref{sec:conclusion} discusses
estimation of the endpoint covariance and spline smoothing.

Throughout, functions and processes are real-valued, the interval is
$[0,1]$, $W$ denotes standard Brownian motion, and $s\wedge t=\min(s,t)$,
$s\vee t=\max(s,t)$. Write
\[
 \bker(s,t)=s\wedge t-st,\qquad
 z(s)=\begin{pmatrix}1-s\\s\end{pmatrix},\qquad
 J=\begin{pmatrix}1&-1\\-1&1\end{pmatrix}.
\]
Here $\bker$ is the covariance of a standard Brownian bridge, and the
Green's function of $-f''$ with Dirichlet conditions $f(0)=f(1)=0$. For a function
$f$, let
\[
 \gamma f=\begin{pmatrix}f(0)\\f(1)\end{pmatrix},\qquad
 d_f=\begin{pmatrix}f'(0)\\-f'(1)\end{pmatrix},\qquad
 (\mathcal E y)(s)=z(s)^\T y.
\]
Here $d_f$ is the vector of inward endpoint derivatives. Endpoint values
are defined on $H^1(0,1)$; derivative traces require $H^2$.
We use $C^\dagger$ for the Moore--Penrose inverse of a symmetric matrix and
$\psd$, $\pd$ for the positive-semidefinite and positive-definite
$2\times2$ matrices.

\begin{definition}\label{def:class}
An ordinary Brownian covariance is the Green's function of a strictly positive
real self-adjoint extension of the minimal operator $-f''$ on
\[
 H^2_0(0,1)=\{f\in H^2:f(0)=f(1)=f'(0)=f'(1)=0\},
\]
that is, of $-f''$ on $(0,1)$ with self-adjoint conditions imposed only at
the endpoints.
A centered Gaussian process with such a covariance is called a
\emph{Brownian process}. Strict positivity means a positive lower spectral
bound. Nonnegative realizations with a null space are treated in
Section~\ref{sec:singular}.
\end{definition}

\section{Complete classification by endpoint covariance}\label{sec:classification}

\begin{theorem}[Endpoint classification]\label{thm:classification}
The map $C\mapsto K_C$, where
\begin{equation}\label{eq:kernel}
 K_C(s,t)=\bker(s,t)+z(s)^\T C z(t),\qquad C\in\psd,
\end{equation}
is a bijection onto the ordinary Brownian covariances. Its inverse reads
off the endpoint values:
\begin{equation}\label{eq:endpoints}
 C_{ij}=K_C(i,j),\qquad i,j\in\{0,1\},
\end{equation}
so $C$ is the covariance matrix of $(X_C(0),X_C(1))$, defined below.
For $C\in\psd$, put $V=\ran C$. Then:
\begin{enumerate}
\item \emph{Process.} The associated Brownian process has the representation
\begin{equation}\label{eq:construction}
 X_C(s)=B(s)+z(s)^\T Y,\qquad
 Y\sim N(0,C),\quad B\ \text{a standard bridge independent of }Y.
\end{equation}
\item \emph{Energy.} $K_C$ is the reproducing kernel of
$\mathcal H_C=\{f\in H^1(0,1):\gamma f\in V\}$ with the quadratic form
\begin{equation}\label{eq:energy}
 q_C[f]=\int_0^1 |(f-\mathcal E\gamma f)'|^2\dd s+(\gamma f)^\T C^\dagger\gamma f
 =\int_0^1|f'|^2\dd s+(\gamma f)^\T(C^\dagger-J)\gamma f.
\end{equation}
\item \emph{Operator.} $K_C$ is the integral kernel of $\mathcal L_C^{-1}$, where
$\mathcal L_Cf=-f''$ on
\begin{equation}\label{eq:domain}
 \dom(\mathcal L_C)=\{f\in H^2(0,1):\gamma f=C(d_f+J\gamma f)\}.
\end{equation}
\end{enumerate}
\end{theorem}

\begin{proof}
\emph{The form.} Fix $C\in\psd$. Every $f\in\mathcal H_C$ decomposes uniquely
as $f=h+\mathcal E y$ with $h\in H^1_0(0,1)$ and $y=\gamma f\in V$. The cross
term between $h'$ and $(\mathcal E y)'$ vanishes and
$\int|(\mathcal E y)'|^2=y^\T Jy$, which gives both expressions in
\eqref{eq:energy}. Since the derivative norm is coercive on $H^1_0$ and
$C^\dagger$ is positive definite on $V$, $q_C$ is closed and coercive, with
norm equivalent to the $H^1$ norm on $\mathcal H_C$; for $V=\{0\}$ it
reduces to $\int_0^1|f'|^2\dd s$ on $H^1_0$.

\emph{The operator.} Integration by parts shows that the operator of $q_C$
acts as $-f''$ with boundary conditions $y\in V$ and
$P_V(d_f+Jy)=C^\dagger y$, where $P_V$ is the orthogonal projection onto $V$.
Multiplying by $C$ gives \eqref{eq:domain}. Conversely,
\eqref{eq:domain} implies $y\in V$; applying $C^\dagger$ then recovers the
projected condition. Hence $\mathcal L_C$ is a strictly
positive self-adjoint realization.

\emph{The kernel and the process.} For fixed $t$, $K_C(\cdot,t)$ has endpoint
vector $Cz(t)$ and Dirichlet component $\bker(\cdot,t)$. For
$f=h+\mathcal E y\in\mathcal H_C$, polarizing \eqref{eq:energy} gives
\[
 q_C\bigl(K_C(\cdot,t),f\bigr)=h(t)+(Cz(t))^\T C^\dagger y=h(t)+z(t)^\T y=f(t).
\]
So $K_C$ is the reproducing kernel of $\mathcal H_C$ and the integral kernel
of $\mathcal L_C^{-1}$. The two independent terms in \eqref{eq:construction}
have covariances $\bker(s,t)$ and $z(s)^\T Cz(t)$, so $X_C$ has covariance $K_C$.

\emph{Completeness and uniqueness.} Let $\mathcal L$ be a strictly positive
real self-adjoint realization and $\mathcal L_{\mathrm D}$ the Dirichlet
realization. For $g\in L^2$, $(\mathcal L^{-1}g-\mathcal L_{\mathrm D}^{-1}g)''=0$,
so the self-adjoint difference has affine range and equals
$\mathcal E C\mathcal E^*$ for a unique symmetric $C$, where
$\mathcal E^*g=\int_0^1z(s)g(s)\dd s$. Thus $\mathcal L^{-1}$ has kernel
\eqref{eq:kernel}. Applying positivity of $\mathcal L^{-1}$ to
$g_\varepsilon=\varepsilon^{-1}\{a_0\mathbf1_{[0,\varepsilon]}+a_1\mathbf1_{[1-\varepsilon,1]}\}$
and letting $\varepsilon\downarrow0$ gives $a^\T Ca\ge0$: the Dirichlet
contribution tends to zero, while $\mathcal E^*g_\varepsilon\to a$. Finally,
$\mathcal L^{-1}$ determines $\mathcal L$, and \eqref{eq:endpoints} determines $C$.
\end{proof}

For $0<t<1$, the column $K_C(\cdot,t)$ is affine on each side of $t$, and
its slope drops by one at $t$. It therefore lies in the energy space
$\mathcal H_C$ but not in $\dom(\mathcal L_C)$, as is usual for Green's
functions. For $t\in\{0,1\}$, the column is affine but does not satisfy the
boundary conditions in \eqref{eq:domain} unless it vanishes. In both cases
it is characterized by the reproducing identity
$q_C(K_C(\cdot,t),f)=f(t)$ from the proof, which holds for every
$t\in[0,1]$.

\subsection{The finite boundary-penalty chart}

When $C\succ0$, there are no hard endpoint constraints: $(X_C(0),X_C(1))$
has a positive density on $\R^2$, and the boundary conditions act only
through the endpoint covariance. Set
\begin{equation}\label{eq:penalty}
 P=C^{-1}-J=\begin{pmatrix}a&c\\c&b\end{pmatrix}.
\end{equation}
Then \eqref{eq:domain} becomes
\[
 f'(0)=af(0)+cf(1),\qquad -f'(1)=cf(0)+bf(1).
\]
The entry $c$ is the only term coupling the two endpoints. When $c=0$
the conditions are \emph{separated}, each involving a single endpoint,
and reduce to the Robin conditions $f'(0)=af(0)$ and $-f'(1)=bf(1)$.
The exact admissibility condition is
\[
 P+J\succ0
 \quad\Longleftrightarrow\quad
 a>-1,\qquad \Delta>0,
\]
where $\Delta=(1+a)(1+b)-(1-c)^2=\det(P+J)$.
As $\Delta\downarrow0$, the endpoint covariance becomes unbounded.
The minimum of $\int|f'|^2$ at endpoint vector $y$ is $y^\T Jy$,
attained by $\mathcal E y$. The reduced energy is therefore
$y^\T(P+J)y$. A boundary penalty may be indefinite while
the full energy is positive. For example, $C=I$ gives
$P=\begin{psmallmatrix}0&1\\1&0\end{psmallmatrix}$, while $P+J=I$.

Matrix inversion in \eqref{eq:penalty} gives
\[
 C=\frac1\Delta
 \begin{pmatrix}1+b&1-c\\1-c&1+a\end{pmatrix}.
\]
Substitution into \eqref{eq:kernel} recovers the formula
\[
 K_C(s,t)=
 \frac{(1-c)+(a+c)(s\wedge t)+(b+c)(1-s\vee t)
       +(ab-c^2)(s\wedge t)(1-s\vee t)}{\Delta}.
\]
Note that the general matrix form \eqref{eq:kernel} remains valid even when
$C$ is singular.

Here $P$ is a boundary contribution to the energy. The joint endpoint
precision is $C^{-1}=P+J$, and the marginal scalar precision of $X_C(0)$
is $1/C_{00}=\Delta/(1+b)$. 

\subsection{Hard endpoint conditions are finite parameters}

The number of unconstrained endpoint coordinates is $\rank C$.
For $C=0$, \eqref{eq:domain} is Dirichlet at both ends. For
$C=\sigma^2 ww^\T$, $\sigma>0$, $w\ne0$, the endpoint vector lies on
$\operatorname{span}\{w\}$ and
\[
 K_C(s,t)=\bker(s,t)+\sigma^2
              \bigl(z(s)^\T w\bigr)\bigl(z(t)^\T w\bigr).
\]
Such singular $C$ imposes hard endpoint constraints while $\mathcal L_C$ remains
strictly positive; it does not give the operator a null space.
Every rank-one $C$ with $w=(\alpha,1)^\T$ arises from the coupled conditions
\begin{equation}\label{eq:coupled}
 f(0)=\alpha f(1),\qquad
 \alpha f'(0)-f'(1)=\kappa f(1)
\end{equation}
with $\sigma^2=1/\{\kappa+(\alpha-1)^2\}$. These conditions have an ordinary
covariance exactly when $\kappa+(\alpha-1)^2>0$, and then
\[
 C=\frac{1}{\kappa+(\alpha-1)^2}
   \begin{pmatrix}\alpha^2&\alpha\\\alpha&1\end{pmatrix},
 \qquad
 K_C(s,t)=\bker(s,t)+
       \frac{\{\alpha(1-s)+s\}\{\alpha(1-t)+t\}}{\kappa+(\alpha-1)^2}.
\]
The remaining direction $w=(1,0)^\T$, with $f(1)=0$, follows by
interchanging the endpoints. The quantity $\kappa+(\alpha-1)^2$ is the
energy of the affine function $g(s)=\alpha(1-s)+s$, which satisfies
the endpoint-value constraint $f(0)=\alpha f(1)$ in \eqref{eq:coupled}.
For example, $\alpha=-1$, $\kappa=0$ is the anti-periodic realization.
Table~\ref{tab:examples} lists selected members of the family.

\begin{table}[tbp]
\centering
\caption{Selected members, with their endpoint covariance matrix $C$.}
\label{tab:examples}
\small
\renewcommand{\arraystretch}{1.2}
\newcommand{\Cstrut}{\vphantom{1+\sigma^2/4}}
\begin{tabular}{lll}
\toprule
Process & Endpoint covariance $C$ & Kernel $K_C(s,t)$ \\
\midrule
Brownian bridge & $\begin{psmallmatrix}\Cstrut 0&0\\\Cstrut 0&0\end{psmallmatrix}$ & $s\wedge t-st$ \\[6pt]
Brownian motion & $\begin{psmallmatrix}\Cstrut 0&0\\\Cstrut 0&1\end{psmallmatrix}$ & $s\wedge t$ \\[6pt]
Time-reversed Brownian motion & $\begin{psmallmatrix}\Cstrut 1&0\\\Cstrut 0&0\end{psmallmatrix}$ & $1-s\vee t$ \\[6pt]
Brownian motion plus random slope & $\begin{psmallmatrix}\Cstrut 0&0\\\Cstrut 0&1+\sigma^2\end{psmallmatrix}$
 & $s\wedge t+\sigma^2st$ \\[6pt]
Anti-periodic Brownian motion & $\begin{psmallmatrix}\Cstrut 1/4&-1/4\\\Cstrut -1/4&1/4\end{psmallmatrix}$
 & $\tfrac14-\tfrac12|s-t|$ \\[6pt]
Slepian Brownian motion & $\begin{psmallmatrix}\Cstrut 1/2&0\\\Cstrut 0&1/2\end{psmallmatrix}$
 & $\tfrac12(1-|s-t|)$ \\[6pt]
Robin--Neumann Brownian motion & $\begin{psmallmatrix}\Cstrut \theta&\theta\\\Cstrut \theta&1+\theta\end{psmallmatrix}$, $\theta>0$
 & $s\wedge t+\theta$ \\[6pt]
\bottomrule
\end{tabular}
\end{table}

\section{Geometry and structural properties}\label{sec:geometry}

\subsection{The covariance cone and its strata}

The linear coordinates
\begin{equation}\label{eq:cone}
 C=\begin{pmatrix}r+x&y\\y&r-x\end{pmatrix}
 \qquad
 C\succeq0\quad\Longleftrightarrow\quad
 r\ge\sqrt{x^2+y^2}
\end{equation}
identify the parameter set with the closed three-dimensional Lorentz
cone. Its eigenvalues are $r\pm\sqrt{x^2+y^2}$. The interior is a smooth
three-dimensional manifold; the nonzero boundary is the two-dimensional
rank-one stratum, and the apex is the bridge. The full family is naturally
stratified; its embedded boundary is not smooth at the apex.

A scale--shape parametrization is $\tau=\tr C=2r$ and, for $\tau>0$,
\begin{equation}\label{eq:disk}
 C=\frac{\tau}{2}
   \begin{pmatrix}1+\xi&\eta\\\eta&1-\xi\end{pmatrix},
 \qquad \xi^2+\eta^2\le1.
\end{equation}
Thus total endpoint variance is a scale and shape lies in a disk.
Its boundary describes one free endpoint coordinate; antipodal endpoint
directions give the same rank-one covariance. Endpoint correlations are
$\eta/\sqrt{1-\xi^2}$ when both endpoint variances are positive.
Unlike correlation coordinates alone, the cone remains unambiguous at
pinned endpoints. Figure~\ref{fig:geometry} shows both coordinate systems.

\begin{figure}[tbp]
\centering
\includegraphics[width=\textwidth]{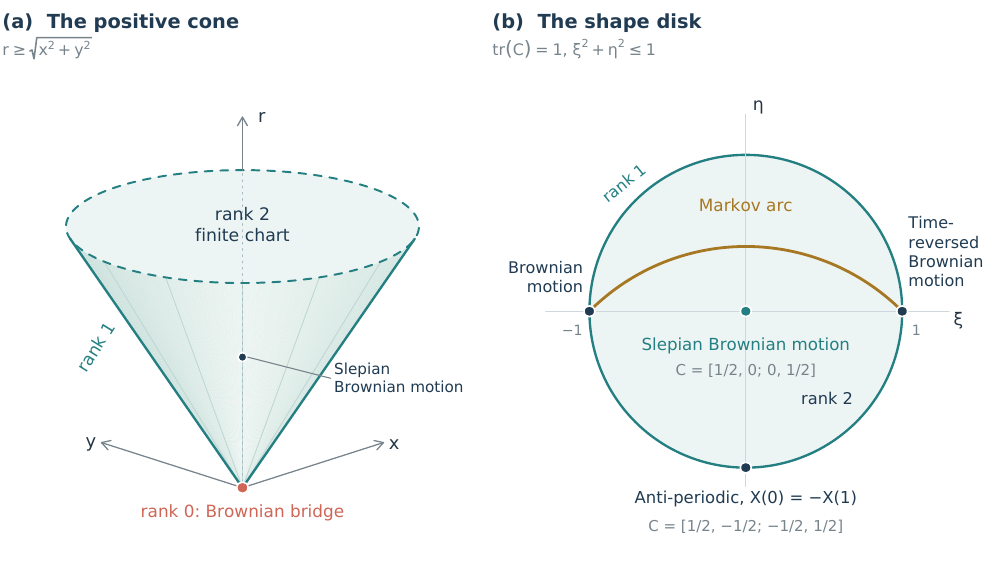}
\caption{Endpoint-covariance geometry. Left: a finite display of the
unbounded cone \eqref{eq:cone}; its side is rank one and its apex is the
bridge. Right: the shape coordinates $(\xi,\eta)$ of \eqref{eq:disk} at $\tr C=1$,
displaying endpoint variance asymmetry and covariance.
The Markov arc is the exact locus $\det C=C_{01}$. The anti-periodic
point is the trace-one member of the anti-periodic family $X(0)=-X(1)$;
the anti-periodic Brownian motion of Table~\ref{tab:examples} is its
trace-$1/2$ member.}
\label{fig:geometry}
\end{figure}

\begin{figure}[tbp]
\centering
\includegraphics[width=\textwidth]{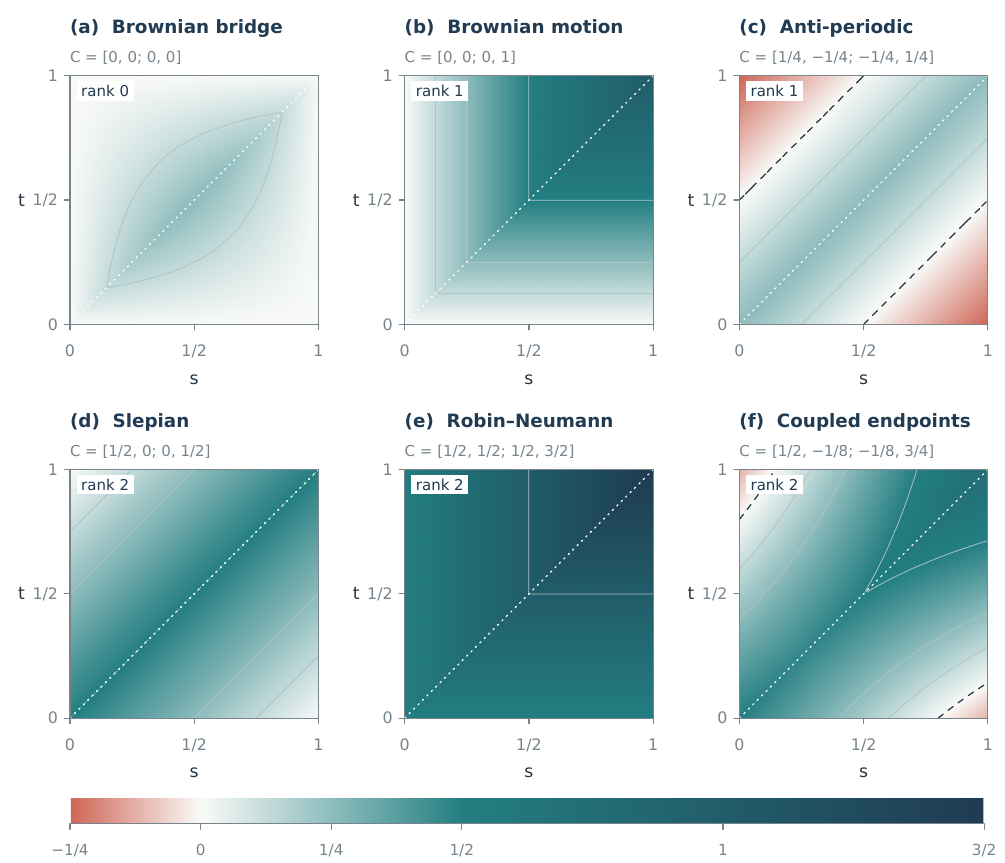}
\caption{Covariance kernels $K_C(s,t)$ of six representatives of
\eqref{eq:kernel}, spanning the rank-zero, rank-one and rank-two cases.
The endpoint covariance is specified in each panel as
$C=[C_{00},C_{01};C_{10},C_{11}]$, row by row. Pinned limits are
plotted from their exact matrices, not approximated by large boundary
penalties. Dashed contours mark zero covariance, and the dotted line is
the diagonal.}
\label{fig:kernels}
\end{figure}

\begin{proposition}[Closure and covariance order]\label{prop:order}
For $C,D\in\psd$,
\[
 \|K_C-K_D\|_\infty=\max_{i,j\in\{0,1\}}|C_{ij}-D_{ij}|,
 \qquad
 K_C\succeq K_D\ \Longleftrightarrow\ C\succeq D.
\]
Moreover, $\rank(K_C-K_D)=\rank(C-D)$ as integral operators.
The family \eqref{eq:kernel} is closed in uniform kernel norm, and
$\{K_C:C\succ0\}$ is dense in it. The bridge is its least element in
covariance order; there is no greatest element.
\end{proposition}

\begin{proof}
The kernel difference is a convex combination of the four entries of
$C-D$, and all four occur at corners. Positivity of $C-D$ gives
positivity of $K_C-K_D$; the reverse implication follows by evaluating
at the endpoints. The map $y\mapsto\mathcal E y$ is injective and its
adjoint is onto, giving the rank identity. A uniform kernel limit
determines a limit of corner matrices by \eqref{eq:endpoints}; that
limit is positive semidefinite. Conversely $C+\varepsilon I\to C$
gives density. Finally, $C\succeq0$ gives $K_C\succeq K_0$, so the bridge is
least. A greatest element $K_G$ would require $G$ to dominate every
positive multiple of $I$, which no matrix does.
\end{proof}

\subsection{Markov and stationary members}

\begin{proposition}[Markov and stationary members]\label{prop:markov}
The process $X_C$ is Markov if and only if
\begin{equation}\label{eq:markov}
 \det C=C_{01}.
\end{equation}
For $C\succ0$, this is exactly the separated Robin condition $c=0$ in
\eqref{eq:penalty}.
The stationary members are 
\begin{equation}\label{eq:stationary}
 C_v=\begin{pmatrix}v+\tfrac14&v-\tfrac14\\[4pt]
                    v-\tfrac14&v+\tfrac14\end{pmatrix},
 \quad v\ge0 \qquad\textrm{with}
 \qquad K_{C_v}(s,t)=v+\tfrac14-\tfrac12|s-t|.
\end{equation}
\end{proposition}

\begin{proof}
The variance is positive at every interior time because
$K_C(t,t)\ge t(1-t)$. For a centered Gaussian process with this property,
the Markov condition is
$K_C(s,u)K_C(t,t)=K_C(s,t)K_C(t,u)$ for $s<t<u$.
Direct expansion of \eqref{eq:kernel} gives the difference
\[
 (C_{01}-\det C)(t-s)(u-t),
\]
which proves \eqref{eq:markov}.
For $C\succ0$, the entry $c$ in \eqref{eq:penalty} is $1-C_{01}/\det C$.
Stationarity requires constant variance. Equating the linear and
quadratic coefficients in $K_C(s,s)$ gives
$C_{11}=C_{00}$ and $C_{01}=C_{00}-1/2$.
Positive semidefiniteness then gives $C_{00}\ge1/4$.
Substitution proves \eqref{eq:stationary} and its sufficiency.
\end{proof}

\emph{Anti-periodic Brownian motion} corresponds to $v=0$ in \eqref{eq:stationary}, and \emph{Slepian Brownian motion}
$(W(s+1)-W(s))/\sqrt2$, a scaled version of the Slepian process
\cite{slepian1961}, corresponds to $v=1/4$. Note that endpoint independence, $C_{01}=0$, differs from separated boundary
conditions, $c=0$ in \eqref{eq:penalty}.
The total variance and increment variance are also immediate:
\begin{align*}
 \tr K_C&=\frac16+\frac{C_{00}+C_{01}+C_{11}}3,\\
 \Var(X_C(t)-X_C(s))
 &=|t-s|+(t-s)^2\bigl(e^\T Ce-1\bigr),
 \qquad e=(1,-1)^\T.
\end{align*}
The explicit construction \eqref{eq:construction} gives continuous, almost surely
nowhere-differentiable paths: only an affine random function has been
added to the bridge. Figure~\ref{fig:kernels} shows five members of
Table~\ref{tab:examples} and one member with coupled endpoints.

\section{Exact representation on any mesh}\label{sec:mesh}

Likelihood computations for $n$ observations of a Gaussian process
generally cost $O(n^3)$ operations, because they involve a dense
covariance matrix. For the ordinary Brownian family, exact likelihoods and posterior means
can be computed in $O(n)$ on any mesh using sparse matrix factorizations
and linear solves, even with independent Gaussian observation noise.
Markussen \cite{markussen2013} approximated the corresponding
log-determinant and predictions by operator calculus, which gives
formulas of linear cost on equidistant grids. Here no approximation is
needed, and the same quantities are computed exactly.

Observe $x_i=X_C(t_i)$ at
$0=t_0<t_1<\cdots<t_n=1$, with spacings $h_i=t_i-t_{i-1}$, and let
$C\succ0$. The precision quadratic form of $(x_0,\ldots,x_n)$ is the discrete
energy $\sum_i(x_i-x_{i-1})^2/h_i$, common to the whole family, plus the
boundary penalty $C^{-1}-J$ acting on the two corner entries. It is
tridiagonal apart from a single endpoint-to-endpoint entry, its
determinant is explicit, and the discretization involves no
approximation. The log-likelihood is therefore
\[
 -\frac{n+1}2\log(2\pi)
 -\frac12\log\det C
 -\frac12\sum_{i=1}^n\log h_i
 -\frac12\sum_{i=1}^n\frac{(x_i-x_{i-1})^2}{h_i}
 -\frac12\begin{pmatrix}x_0\\x_n\end{pmatrix}^{\!\T}\!(C^{-1}-J)
          \begin{pmatrix}x_0\\x_n\end{pmatrix},
\]
which has computational cost $O(n)$ and depends on $C$ only through a
$2\times2$ matrix. The same structure gives linear-cost sampling.
The displayed likelihood assumes observations at both interval endpoints.
If the first and last observation times are $a<b$, the same construction
applies on $[a,b]$, with endpoint covariance
$\Sigma_{ab}=(K_C(u,v))_{u,v\in\{a,b\}}$. Conditional on these first and
last observed values, the intervening observations form a Brownian bridge
whose conditional law does not depend on $C$.

For $x=(x_0,\ldots,x_n)^\T$, let $D_nx=(x_1-x_0,\ldots,x_n-x_{n-1})^\T$
be the vector of increments, so that $(D_nx)_i/h_i$ is the slope of the
piecewise-linear interpolant on the $i$th interval, and let
$R_nx=(x_0,x_n)^\T$ select the endpoint values, the discrete analogue
of $\gamma$.

\begin{theorem}[Mesh exactness]\label{thm:mesh}
For $C\succ0$, the precision matrix
\begin{equation}\label{eq:mesh-precision}
 Q_n=D_n^\T\diag(h_1^{-1},\ldots,h_n^{-1})D_n
          +R_n^\T(C^{-1}-J)R_n
\end{equation}
satisfies
\begin{equation}\label{eq:mesh-exact}
 (Q_n^{-1})_{ij}=K_C(t_i,t_j),\qquad
 \det Q_n=\frac{\det(C^{-1})}{\prod_{i=1}^n h_i}.
\end{equation}
For general $C\succeq0$, the sampled law is supported on
$\{x:R_nx\in\ran C\}$ and has rank $n-1+\rank C$.
Its precision form on that subspace is $q_C$ evaluated on the
piecewise-linear interpolant.
\end{theorem}

\begin{proof}
The derivative energy of the nodal interpolant is
$\sum_i(x_i-x_{i-1})^2/h_i$. Its endpoint term is exactly the second
term of \eqref{eq:mesh-precision}, without mesh rescaling.
Every column $K_C(\cdot,t_j)$ of the Green's function lies in the piecewise-linear trial
space and represents nodal evaluation in the energy inner product.
Restricting the reproducing identity from
Theorem~\ref{thm:classification} therefore gives \eqref{eq:mesh-exact}'s
inverse identity, with no approximation.

For the determinant, the interior block of $Q_n$ is the Dirichlet
precision, with determinant $(\prod_i h_i)^{-1}$ by the tridiagonal
determinant recurrence, and the endpoint Schur complement of $Q_n$ is
$C^{-1}$; for $n=1$, $Q_1=C^{-1}$ directly. The representation \eqref{eq:construction} at the nodes gives
rank $n-1+\rank C$ and the restricted precision statement for
singular $C$.
\end{proof}

For $n\ge2$, a nonzero $c$ in \eqref{eq:penalty} makes the dependence graph of the nodes a cycle rather than a chain,
so the sampled process is not a Markov chain in the time ordering, consistent with 
Proposition~\ref{prop:markov}.

\section{Nonnegative operators with a null space}\label{sec:singular}

Hard endpoint constraints in $C$ leave $\mathcal L_C$ strictly positive.
Nonnegative realizations with a null space have no ordinary inverse;
they instead admit the orthogonally normalized pseudoinverses described
below.

\begin{proposition}[All nonnegative realizations]\label{prop:nonnegative}
Every real nonnegative self-adjoint extension is uniquely described by
a subspace $V\subseteq\R^2$ and a positive-semidefinite operator
$A:V\to V$. Its closed form is
\begin{equation}\label{eq:semiform}
 q_{V,A}[h+\mathcal E y]=\int_0^1|h'|^2\dd s+y^\T Ay,
 \qquad h\in H^1_0(0,1),\quad y\in V.
\end{equation}
Its operator acts as $-f''$ on
\[
 \dom(\mathcal L_{V,A})=
 \{f\in H^2:\gamma f\in V,\ 
                P_V(d_f+J\gamma f)=A\gamma f\}.
\]
Its null space is
\begin{equation}\label{eq:null}
 \mathcal N=\mathcal E(\ker A).
\end{equation}
If $A$ is positive definite on $V$ (vacuously for $V=\{0\}$), it is the realization of
Theorem~\ref{thm:classification} with $C=A^{-1}$ on $V$ and zero on
$V^\perp$. If $A$ is singular, the nullity is one or two.
\end{proposition}

\begin{proof}
A self-adjoint endpoint domain is a maximal isotropic plane for
\[
 \omega\bigl((y,d),(\widetilde y,\widetilde d)\bigr)
     =y^\T\widetilde d-d^\T\widetilde y.
\]
Let $V$ be its projection onto the value coordinate.
The vertical pairs are exactly $\{0\}\times V^\perp$: isotropy gives
the inclusion, and dimensions give equality. Consequently
$P_Vd=My$ for a uniquely determined symmetric operator $M$ on $V$.
Conversely these conditions define a maximal isotropic plane.

The corresponding form on $\{f\in H^1:\gamma f\in V\}$ is
$\int|f'|^2+(\gamma f)^\T M\gamma f$. The trace inequality
\[
 |\gamma f|^2\le\varepsilon\|f'\|_2^2+
                    c_\varepsilon\|f\|_2^2,\qquad\varepsilon>0,
\]
shows that adding a sufficiently large multiple of $\|f\|_2^2$
gives a norm equivalent to the $H^1$ norm on this closed trace
subspace. Thus the form is closed and bounded below; integration
by parts identifies its operator domain. Decomposing
$f=h+\mathcal E y$ gives \eqref{eq:semiform} with
$A=M+P_VJ|_V$. It is nonnegative exactly when $A\succeq0$, and its
zero-energy functions are exactly \eqref{eq:null}.
The strictly positive case follows by comparison with
\eqref{eq:energy}.
\end{proof}

\begin{proposition}[Proper covariance of the pseudoinverse]\label{prop:pseudo}
Extend $A^\dagger$ by zero on $V^\perp$, and let $\Pi$ be the $L^2$
orthogonal projection onto $\mathcal N^\perp$. The Moore--Penrose inverse of
the realization in Proposition~\ref{prop:nonnegative} has kernel
\[
 K^\dagger=\Pi_s\Pi_t\!
       \left[\bker(s,t)+z(s)^\T A^\dagger z(t)\right].
\]
The subscripts apply $\Pi$ in the respective kernel variables.
This is a continuous positive-semidefinite covariance kernel. We also call
its centered Gaussian process $X^\dagger$ a \emph{Brownian process}. It has
continuous sample paths and
satisfies $\int_0^1X^\dagger(s)\nu(s)\dd s=0$ almost surely
for every $\nu\in\mathcal N$. For $0<t<1$, in distributions on $(0,1)$,
\begin{equation}\label{eq:corrected-source}
 -\partial_s^2K^\dagger(s,t)=\delta(s-t)
                   -\sum_{j=1}^{\dim\mathcal N}\nu_j(s)\nu_j(t),
\end{equation}
where $(\nu_j)$ is any $L^2$-orthonormal basis of $\mathcal N$.
\end{proposition}

\begin{proof}
Put $C=A^\dagger$ and $S=\ran A$. On the subspace with endpoint values
in $S$, the form $q_{V,A}$ is exactly $q_C$.
For $g\perp\mathcal N$, the function $u_0=K_Cg$ solves
$q_{V,A}(u_0,f)=\langle g,f\rangle$ for endpoint values in $S$.
Every remaining test function differs from such a function by an
element of $\mathcal N$, on which both sides vanish. Subtracting the null-space
projection of $u_0$ therefore gives the unique solution in $\mathcal N^\perp$.
Projecting the input as well proves $\mathcal L^\dagger=\Pi K_C\Pi$.

The same expression is the covariance of $\Pi X_C$, so $X^\dagger$ can be
taken to be $\Pi X_C$. Since $\Pi$ subtracts from $X_C$ at most two
affine functions with Gaussian coefficients, $X^\dagger$ has continuous
paths, $K^\dagger$ is continuous, and $\int_0^1X^\dagger\nu=0$ for every
$\nu\in\mathcal N$. Finally
$\mathcal L\mathcal L^\dagger=\Pi$ gives \eqref{eq:corrected-source}.
\end{proof}

Here $\Pi f=f-\sum_j\nu_j\int\nu_jf$ is explicit and involves at most two
affine modes. In the two classical cases below, the null space consists
of the constant functions, so $\nu_1=1$ and $\Pi$ simply subtracts the
time average $\int_0^1f$.

\begin{example}[Neumann--Neumann Brownian motion]\label{ex:nn}
The Neumann conditions $f'(0)=f'(1)=0$ leave both endpoints free; they
are the no-flux conditions for diffusion along an insulated rod. In
Proposition~\ref{prop:nonnegative} they correspond to $V=\R^2$
and $A=J$. Adding a constant to $f$ changes neither the energy
$\int_0^1|f'|^2$ nor the boundary conditions, so the constants form the
null space and no ordinary inverse exists. The corresponding formal
Gaussian model is improper because its constant mode is unpenalized.
Proposition~\ref{prop:pseudo} removes the constant mode, with $A^\dagger=J/4$, and
gives
\[
 K_{\mathrm{NN}}(s,t)=\frac{s^2+t^2}{2}-s\vee t+\frac13,
 \qquad
 X_{\mathrm{NN}}(s)=W(s)-\int_0^1W(u)\dd u.
\]
Neumann--Neumann Brownian motion is thus Brownian motion minus its time
average, the mean-centered Brownian motion of \cite{raketbestiary}.
Centering time-reversed Brownian motion gives the same law, as it must:
the conditions are unchanged under $s\mapsto1-s$. The variance is $1/3$
at both endpoints and $1/12$ at the midpoint, every path has zero time
average, and the kernel is \emph{balanced},
$\int_0^1K_{\mathrm{NN}}(s,t)\dd s=0$, in the terminology of
\cite{raket2021}.
\end{example}

\begin{example}[Periodic Brownian motion]\label{ex:per}
The periodic conditions $f(0)=f(1)$, $f'(0)=f'(1)$ join the two
endpoints smoothly, so the interval becomes a circle of circumference
one. In Proposition~\ref{prop:nonnegative} they correspond to
$V=\operatorname{span}\{(1,1)^\T\}$ and $A=0$. As in
Example~\ref{ex:nn}, no point of the circle fixes the level, and the
constants again form the null space. Here $A^\dagger=0$, and
Proposition~\ref{prop:pseudo} gives
\[
 K_{\mathrm{per}}(s,t)=\frac{|s-t|^2}{2}-\frac{|s-t|}{2}+\frac1{12},
 \qquad
 X_{\mathrm{per}}(s)=B(s)-\int_0^1B(u)\dd u,
\]
the Brownian bridge minus its time average. Its paths satisfy
$X_{\mathrm{per}}(0)=X_{\mathrm{per}}(1)$ and close into loops. The
process is stationary on the circle: its variance is $1/12$ everywhere,
and its covariance depends only on the distance between $s$ and $t$
measured around the circle, since it is unchanged under
$|s-t|\mapsto1-|s-t|$. 
\end{example}

The same construction applies when the null space contains all affine
functions. This happens for $V=\R^2$ and $A=0$, with boundary conditions
$f'(0)=f'(1)=f(1)-f(0)$, under which the energy measures only the
deviation of $f'$ from the chord slope $f(1)-f(0)$.
Proposition~\ref{prop:pseudo} then projects the bridge covariance away
from both constants and linear functions. The resulting process is
Brownian motion with its least-squares line removed.

\section{Consequences and scope}\label{sec:conclusion}

The classification concerns the real operator $-f''$ on $[0,1]$ with
self-adjoint endpoint conditions. Strictly positive realizations give
the ordinary covariances \eqref{eq:kernel}; noninvertible nonnegative
realizations require the normalized pseudoinverses of
Section~\ref{sec:singular}.

For the ordinary family, estimating the boundary conditions amounts
to estimating $C$. With noiseless observations at both endpoints,
each curve contributes a single draw from $N(0,C)$: conditional on
these endpoint values, the remaining observations carry no further
information about $C$. For $M$ independent centered curves sharing
the same endpoint covariance, write
$Y_m=(x_{m,0},x_{m,n})^\T$. The maximum likelihood estimate over
$C\succ0$ is
\[
 \widehat C=\frac{1}{M}\sum_{m=1}^M Y_mY_m^\T,
\]
provided $\widehat C$ is positive definite. Thus, once both endpoints
are observed, replication rather than finer sampling supplies
additional information about $C$. If one or both original endpoints
are unobserved, the first and last observed values remain sufficient
for $C$, provided there are at least two distinct noiseless
observation times; their covariance generally differs from $C$.
With measurement noise, observations between them can also contribute.

The same endpoint covariance controls the boundary behavior of a
smoother. Minimizing
\[
 \sum_i\{y_i-f(t_i)\}^2+\lambda q_C[f]
 \quad\text{over }\mathcal H_C,\qquad \lambda>0,
\]
gives a continuous piecewise-affine fit with knots among the
observation times. Under independent $N(0,\sigma^2)$ observation
errors, independent of the prior process, this is the posterior mean
for a centered Gaussian prior with covariance
$(\sigma^2/\lambda)K_C$
\cite{kimeldorf1970,rasmussen2006}.
The choice $C=0$ pins both ends to zero; more generally, singular
$C$ restricts the endpoint vector to $\ran C$.
Within a fixed endpoint support, increasing $C$ in covariance order
weakens the endpoint penalty.

Across the ordinary family, the Brownian bridge component remains
the same. What changes is how the endpoints are allowed to vary,
and how they are tied to one another.


\begin{thebibliography}{9}
\bibitem{kimeldorf1970}
G.~S.~Kimeldorf and G.~Wahba.
A correspondence between Bayesian estimation on stochastic processes
and smoothing by splines.
\emph{Annals of Mathematical Statistics}, 41(2):495--502, 1970.

\bibitem{markussen2013}
B.~Markussen.
Functional data analysis in an operator-based mixed-model framework.
\emph{Bernoulli}, 19(1):1--17, 2013.

\bibitem{raket2021}
L.~L.~Raket.
Differential equations, splines and Gaussian processes.
arXiv:2102.03306, 2021.
\url{https://arxiv.org/abs/2102.03306}.

\bibitem{raketbestiary}
L.~L.~Raket.
A Bestiary of Brownian Processes: A Field Guide to Boundary Conditions
and Conditioning.
arXiv:2610.33747, 2026.
\url{https://arxiv.org/abs/2610.33747}.

\bibitem{rasmussen2006}
C.~E.~Rasmussen and C.~K.~I.~Williams.
\emph{Gaussian Processes for Machine Learning}.
MIT Press, Cambridge, MA, 2006.

\bibitem{slepian1961}
D.~Slepian.
First passage time for a particular Gaussian process.
\emph{Annals of Mathematical Statistics}, 32(2):610--612, 1961.

\bibitem{zettl2005}
A.~Zettl.
\emph{Sturm--Liouville Theory}.
Mathematical Surveys and Monographs, vol.~121.
American Mathematical Society, Providence, RI, 2005.
\end{thebibliography}
\end{document}